\documentclass[runningheads,a4paper]{llncs}

\usepackage{amssymb}
\usepackage{graphicx}

\usepackage{url}
\urldef{\mailsa}\path|{aurelian}@sci.kuniv.edu.kw|    
\newcommand{\keywords}[1]{\par\addvspace\baselineskip
\noindent\keywordname\enspace\ignorespaces#1}

\begin{document}

\mainmatter  

\title{A new class of interpolatory $L$-splines\\
with adjoint end conditions}

\titlerunning{A new class of interpolatory $L$-splines}

%
%
\author{Aurelian Bejancu%
\thanks{Corresponding author.}%
\and Reyouf S. Al-Sahli}
\authorrunning{A new class of interpolatory $L$-splines}

\institute{Kuwait University, Department of Mathematics,\\
PO Box 5969, Safat 13060, Kuwait\\
\mailsa
}

%
%

\toctitle{A new class of interpolatory $L$-splines}
\tocauthor{Aurelian Bejancu and Reyouf S. Al-Sahli}
\maketitle

\begin{abstract}
A thin plate spline for interpolation of smooth transfinite data
prescribed along concentric circles was recently proposed by Bejancu, using
Kounchev's polyspline method. The construction of the new `Beppo Levi
polyspline' surface reduces, via separation of variables, to that of a
countable family of univariate $L$-splines, indexed by the frequency integer
$k$. This paper establishes the existence, uniqueness and variational
properties of the `Beppo Levi $L$-spline' schemes corresponding to non-zero
frequencies $k$. In this case, the resulting $L$-spline end conditions are
formulated in terms of \emph{adjoint} differential operators, unlike the usual
`natural' $L$-spline end conditions, which employ identical operators at both
ends. Our $L$-spline error analysis leads to an $L^{2}$-error bound for
transfinite surface interpolation with Beppo Levi polysplines.
\keywords{interpolation, $L$-spline, Beppo Levi polyspline, approximation order}
\end{abstract}

\section{Introduction}

The \emph{thin plate spline} (TPS) interpolant for scattered data 
was defined by Duchon \cite{Du76} as the unique minimizer of the squared
seminorm%
\begin{equation}
\left\Vert F\right\Vert _{BL}^{2}:=\int\!\!\int_{\mathbb{R}^{2}}\left(
\left\vert F_{xx}\right\vert ^{2}+2\left\vert F_{xy}\right\vert ^{2}%
+\left\vert F_{yy}\right\vert ^{2}\right)  \mathrm{d}x\,\mathrm{d}y,
\label{eq:Duchon}%
\end{equation}
subject to $F$ taking prescribed values at a finite number of scattered
locations. The minimization takes place in the \emph{Beppo Levi} space of
continuous functions $F$ with generalized second-order partial derivatives in
$L^{2}\left(  \mathbb{R}^{2}\right)  $.

Recently, Bejancu \cite{MMA12} proposed a new type of TPS surface, passing
through several continuous curves prescribed along concentric circles. The new
surface minimizes, for $F\in C^{2}\left(  \mathbb{R}^{2}\backslash\left\{
0\right\}  \right)  $, the polar coordinate version of (\ref{eq:Duchon}):%
\begin{equation}
\left\Vert f\right\Vert _{BL}^{2}:=\int_{0}^{\infty}\!\int_{-\pi}^{\pi
}\left\{  \left\vert f_{rr}\right\vert ^{2}+2\left\vert \frac{f_{\theta}%
}{r^{2}}-\frac{f_{\theta r}}{r}\right\vert ^{2}+\left\vert \frac
{f_{\theta\theta}}{r^{2}}+\frac{f_{r}}{r}\right\vert ^{2}\right\}
r\,\mathrm{d}\theta\,\mathrm{d}r, \label{eq:Du-polar}%
\end{equation}
where $f\left(  r,\theta\right)  :=F\left(  r\cos\theta,r\sin\theta\right)  $
denotes the polar form of $F$. Similar surfaces for \emph{transfinite}
interpolation have also been studied in \cite{JAT08,CA11} in the case of
continuous periodic data prescribed along parallel lines or hyperplanes (see
also the survey \cite{IMA13}).

The `transfinite TPS' surfaces belong to the class of multivariate
\emph{polysplines} introduced by Kounchev \cite{Ku01}. In the context of data
prescribed on concentric circles $r=r_{j}$, $j\in\left\{  1,\ldots,n\right\}
$, with $0<r_{1}<\ldots<r_{n}$, let us denote $\rho:=\left\{  r_{1}%
,\ldots,r_{n}\right\}  $ and $\Omega:=\left\{  \left(  r,\theta\right)
:r_{1}\leq r\leq r_{n},\ -\pi\leq\theta\leq\pi\right\}  $. A function
$S:\Omega\rightarrow\mathbb{R}$ is termed a \emph{biharmonic polyspline }on
annuli determined by $\rho$ if two conditions hold: first, $S$ and its polar
form $s$ are piecewise biharmonic, \emph{i.e.}%
\[
\left(  \partial_{xx}+\partial_{yy}\right)  ^{2}S\left(  x,y\right)  =\left(
\partial_{rr}+r^{-1}\partial_{r}+r^{-2}\partial_{\theta\theta}\right)
^{2}s\left(  r,\theta\right)  =0,
\]
on each annulus $r_{j}<r<r_{j+1}$, $-\pi\leq\theta\leq\pi$, for $1\leq j\leq
n-1$; and second, $S\in C^{2}\left(  \Omega\right)  $, \emph{i.e.}
neighbouring pieces join up $C^{2}$-continuously across the interface circles.
For sufficiently smooth periodic data functions $u,v,\mu_{j}:\left[  -\pi
,\pi\right]  \rightarrow\mathbb{R}$, $1\leq j\leq n$, Kounchev proved that
such a polyspline surface is uniquely determined by transfinite interpolation
conditions%
\begin{equation}
s\left(  r_{j},\theta\right)  =\mu_{j}\left(  \theta\right)  ,\quad
\forall\theta\in\left[  -\pi,\pi\right]  ,\ \forall j\in\left\{
1,\ldots,n\right\}  , \label{eq:transfinite}%
\end{equation}
together with boundary conditions $\partial_{r}s\left(  r_{1},\theta\right)
=u\left(  \theta\right)  $ and $\partial_{r}s\left(  r_{n},\theta\right)
=v\left(  \theta\right)  $, $\forall\theta\in\left[  -\pi,\pi\right]  $. He
also extended this result to polysplines of higher orders and more general
interface configurations in $\mathbb{R}^d$. In the case of \emph{cardinal} 
interpolation at the bi-infinite set of hyperspheres of radii $r=e^j$, 
$j\in\mathbb{Z}$, Kounchev and Render \cite{KR05} constructed
polysplines that satisfy growth conditions as $r\rightarrow 0$ and 
$r\rightarrow\infty$.

In \cite{MMA12}, Bejancu proposed a global polyspline $S:\mathbb{R}%
^{2}\rightarrow\mathbb{R}$ for which boundary conditions on the above extreme
circles $r=r_{1}$ and $r=r_{n}$ are replaced by the requirement that the polar
Beppo Levi energy (\ref{eq:Du-polar}) is finite for $f:=s$. This \emph{Beppo
Levi polyspline} has two additional biharmonic pieces over the extreme annuli
$0<r<r_{1}$ and $r>r_{n}$, such that $S\in C^{2}\left(  \mathbb{R}%
^{2}\backslash\{0\}\right)  $. The new surface is automatically continuous at
$0$, but its partial derivatives can have a singularity at $0$.

For sufficiently smooth data, it turns out that there exists a one-parameter
family of such Beppo Levi polysplines on annuli determined by $\rho$, each
satisfying the transfinite interpolation conditions (\ref{eq:transfinite}).
Two surfaces $S^{A}$ and $S^{B}$ of this family are uniquely determined in
\cite[Theorem~1]{MMA12} by the following additional conditions: $S^{A}$ takes
an arbitrarily prescribed value at $0$, while $S^{B}$ is biharmonic at $0$
(hence, non-singular). Both $S^{A}$ and $S^{B}$ are then characterized as
genuine TPS surfaces, \emph{i.e.}\ minimizers of (\ref{eq:Du-polar}), subject
to their respective interpolation conditions.

Following the method of separation of variables used by Kounchev \cite{Ku01},
the construction of the Beppo Levi polysplines $S^{A}$, $S^{B}$ is obtained in
\cite[section~4]{MMA12} via the absolutely convergent Fourier representation
in polar form%
\begin{equation}
s\left(  r,\theta\right)  =\sum_{k\in\mathbb{Z}}\widehat{s}_{k}\left(
r\right)  e^{ik\theta},\quad\left(  r,\theta\right)  \in\left[  0,\infty
\right)  \times\left[  -\pi,\pi\right]  . \label{eq:F-series}%
\end{equation}
For each frequency $k$, the amplitude coefficient $\widehat{s}_{k}$ of this
representation is a univariate $L_{k}$\emph{-spline} for an ordinary
differential operator operator $L_{k}$, as described in the next section.
Moreover, the form of $\widehat{s}_{k}$ on the extreme intervals $\left(
0,r_{1}\right)  $ and $\left(  r_{n},\infty\right)  $ is determined by the
condition that the corresponding Plancherel component of the Beppo Levi energy
(\ref{eq:Du-polar}) of $s$ is finite.

The present paper studies the class of such \emph{Beppo Levi }$L_{k}%
$\emph{-splines} corresponding to non-zero frequencies $k$ (see section 2). In
this case, the restrictions satisfied by $\widehat{s}_{k}$ on the extreme
intervals exhibit a twisted symmetry, expressed in terms of adjoint
differential operators. Different features appear in the radial case $k=0$,
treated in the companion paper \cite{RTPS}, which is connected to Rabut's work
on radially symmetric thin plate splines \cite{Rab96}.

In section 3, we prove the existence, uniqueness and variational
characterization of interpolation schemes with Beppo Levi $L_{k}$-splines, as
required by the construction of \cite{MMA12}. A part of these results,
corresponding to $\left\vert k\right\vert \geq2$, has first been obtained in
the MSc thesis \cite{MSc}. For $\left\vert k\right\vert \geq2$, we also provide a 
linear representation of Beppo Levi $L_{k}$-splines in terms of dilates of a basis 
function related to the generalized Whittle-Mat\'{e}rn-Sobolev 
kernels introduced by Bozzini, Rossini, and Schaback \cite{BRS13}.
Further, in section 4, we apply an error analysis
of the Beppo Levi $L_{k}$-spline schemes to establish the $L^{2}%
$-approximation order $O\left(  h^{2}\right)  $ for transfinite surface
interpolation with Beppo Levi polysplines on annuli, where $h$ is the maximum
distance between successive interface circles. The extension of this work to
higher order Beppo Levi polysplines on annuli and their $L$-spline Fourier
coefficients will be addressed in a separate paper.

\section{Preliminaries}

\subsection{Energy spaces}

For each $r\geq0$, define the Fourier coefficients of $f\left(  r,\theta
\right)  $ with respect to $\theta$ by%
\begin{equation}
\widehat{f}_{k}\left(  r\right)  :=\frac{1}{2\pi}\int_{-\pi}^{\pi}%
e^{-ik\theta}f\left(  r,\theta\right)  d\theta,\quad k\in\mathbb{Z}.
\label{eq:F-coeffs}%
\end{equation}
The following observation shows the effect on $\widehat{f}_{k}$ of the
condition that the polar Beppo Levi integral (\ref{eq:Du-polar}) is finite.
Namely, if $f$ is the polar form of $F\in C^{2}\left(  \mathbb{R}%
^{2}\backslash\left\{  0\right\}  \right)  $, then $\widehat{f}_{k}\in
C^{2}\left(  0,\infty\right)  $ and Plancherel's formula implies the identity
\cite[(3.2)]{MMA12}:%
\begin{equation}
\left\Vert f\right\Vert _{BL}^{2}=2\pi\sum_{k\in\mathbb{Z}}\left\Vert
\widehat{f}_{k}\right\Vert _{k}^{2}, \label{eq:Plancherel}%
\end{equation}
where, for each $k\in\mathbb{Z}$, we denote
\begin{equation}
\left\Vert \psi\right\Vert _{k}^{2}:=\int_{0}^{\infty}\left\{  \left\vert
\frac{\mathrm{d}^{2}\psi}{\mathrm{d}r^{2}}\right\vert ^{2}+2k^{2}\left\vert
\frac{\psi}{r^{2}}-\frac{1}{r}\frac{\mathrm{d}\psi}{\mathrm{d}r}\right\vert
^{2}+\left\vert k^{2}\frac{\psi}{r^{2}}-\frac{1}{r}\frac{\mathrm{d}\psi
}{\mathrm{d}r}\right\vert ^{2}\right\}  r\,\mathrm{d}r. \label{eq:k-norm}%
\end{equation}

Let $AC_{\mathrm{loc}}\left(  0,\infty\right)  $ be the vector space of
functions $\psi:\left(  0,\infty\right)  \rightarrow\mathbb{C}$ that are
absolutely continuous on any interval $\left[  a,b\right]  $, $0<a<b<\infty$.
We denote by $\Lambda_{1}$ the vector space of functions 
$\psi\in C^{1}\left(  0,\infty\right)$ with 
$\psi^{\prime}\in AC_{\mathrm{loc}}\left(  0,\infty\right)  $, such that
$r^{1/2}\psi^{\prime\prime}$ and $r^{-1/2}\psi^{\prime}-r^{-3/2}\psi$ belong
to $L^{2}\left(  0,\infty\right)  $. Also, by $\Lambda_{2}$ we denote the vector 
space of functions $\psi\in C^{1}\left(  0,\infty\right)  $ with $\psi^{\prime}\in
AC_{\mathrm{loc}}\left(  0,\infty\right)  $, such that $r^{1/2}\psi
^{\prime\prime}$, $r^{-1/2}\psi^{\prime}$, and $r^{-3/2}\psi$ all belong to
$L^{2}\left(  0,\infty\right)  $. Note that $\left\Vert \cdot\right\Vert _{k}$
is a norm on $\Lambda_{2}$ for $\left\vert k\right\vert \geq2$ and a semi-norm 
on $\Lambda_{1}$ for $k=\pm1$. The results of section 3 employ the following 
properties of functions from the spaces $\Lambda_{1}$ and $\Lambda_{2}$.

\begin{lemma}
\label{le:aux}\emph{(i)} If $\psi\in\Lambda_{2}$, there exist non-negative
constants $C_{\psi}$ and $\widetilde{C}_{\psi}$, such that%
\begin{equation}%
\begin{array}
[c]{l}%
\left\vert \psi\left(  r\right)  \right\vert \leq C_{\psi}\left(
r^{3/2}+r\left\vert 1-r\right\vert ^{1/2}\right)  ,\\
\left\vert \psi^{\prime}\left(  r\right)  \right\vert \leq\widetilde{C}_{\psi
}\left(  r^{1/2}+\left\vert 1-r\right\vert ^{1/2}\right)  ,
\end{array}
\quad\forall r>0. \label{eq:aux2}%
\end{equation}

\emph{(ii)} If $\psi\in\Lambda_{1}$, there exist non-negative constants
$C_{\psi}$ and $\widetilde{C}_{\psi}$, such that%
\begin{equation}%
\begin{array}
[c]{l}%
\left\vert \psi\left(  r\right)  \right\vert \leq C_{\psi}r\left(
1+\left\vert \ln r\right\vert ^{1/2}\right)  ,\\
\left\vert \psi^{\prime}\left(  r\right)  \right\vert \leq\widetilde{C}_{\psi
}\left(  1+\left\vert \ln r\right\vert ^{1/2}\right)  ,
\end{array}
\quad\forall r>0. \label{eq:aux1}%
\end{equation}

\end{lemma}

\begin{proof}
(i) For each $r>0$, we use the Leibniz-Newton formula%
\begin{eqnarray*}
r^{-3/2}\psi\left(  r\right)  -\psi\left(  1\right)   &  = & \int_{1}^{r}\left[
t^{-3/2}\psi\left(  t\right)  \right]  ^{\prime}\mathrm{d}t\\
&  = & \int_{1}^{r}\left[  t^{-3/2}\psi^{\prime}\left(  t\right)  -\frac{3}%
{2}t^{-5/2}\psi\left(  t\right)  \right]  \mathrm{d}t.
\end{eqnarray*}
Via Cauchy-Schwarz, the last integral is bounded above in modulus
by%
\begin{eqnarray*}
\lefteqn{  \left\vert \int_{1}^{r}t^{-1}\left[  t^{-1/2}\psi^{\prime}\left(  t\right)
\right]  \mathrm{d}t\right\vert +\frac{3}{2}\left\vert \int_{1}^{r}%
t^{-1}\left[  t^{-3/2}\psi\left(  t\right)  \right]  \mathrm{d}t\right\vert } \\
&  \leq & \left\vert \int_{1}^{r}t^{-2}\mathrm{d}t\right\vert ^{\frac{1}{2}%
}\left\{  \left\vert \int_{1}^{r}\left\vert t^{-1/2}\psi^{\prime}\left(
t\right)  \right\vert ^{2}\mathrm{d}t\right\vert ^{\frac{1}{2}}+\frac{3}%
{2}\left\vert \int_{1}^{r}\left\vert t^{-3/2}\psi\left(  t\right)  \right\vert
^{2}\mathrm{d}t\right\vert ^{\frac{1}{2}}\right\} \\
&  \leq & \left\vert 1-r^{-1}\right\vert ^{\frac{1}{2}}\left\{  \left\Vert
r^{-1/2}\psi^{\prime}\right\Vert _{L^{2}\left(  0,\infty\right)  }+\left(
3/2\right)  \left\Vert r^{-3/2}\psi\right\Vert _{L^{2}\left(  0,\infty\right)
}\right\}
\end{eqnarray*}
which implies the first of inequalities (\ref{eq:aux2}). For the second
inequality, we similarly start with%
\begin{eqnarray*}
r^{-1/2}\psi^{\prime}\left(  r\right)  -\psi^{\prime}\left(  1\right)   &
= & \int_{1}^{r}\left[  t^{-1/2}\psi^{\prime}\left(  t\right)  \right]  ^{\prime
}\mathrm{d}t\\
&  = & \int_{1}^{r}\left[  t^{-1/2}\psi^{\prime\prime}\left(  t\right)  -\frac
{1}{2}t^{-3/2}\psi^{\prime}\left(  t\right)  \right]  \mathrm{d}t,
\end{eqnarray*}
which holds for each $r>0$, since $\psi^{\prime}\in AC_{\mathrm{loc}}\left(
0,\infty\right)  $. Hence, we obtain the following upper bound on the modulus
of last integral:%
\begin{eqnarray*}
\lefteqn{  \left\vert \int_{1}^{r}t^{-1}\left[  t^{1/2}\psi^{\prime\prime}\left(
t\right)  \right]  \mathrm{d}t\right\vert +\frac{1}{2}\left\vert \int_{1}%
^{r}t^{-1}\left[  t^{-1/2}\psi^{\prime}\left(  t\right)  \right]
\mathrm{d}t\right\vert } \\
&  \leq & \left\vert \int_{1}^{r}t^{-2}\mathrm{d}t\right\vert ^{\frac{1}{2}%
}\left\{  \left\vert \int_{1}^{r}\left\vert t^{1/2}\psi^{\prime\prime}\left(
t\right)  \right\vert ^{2}\mathrm{d}t\right\vert ^{\frac{1}{2}}+\frac{1}%
{2}\left\vert \int_{1}^{r}\left\vert t^{-1/2}\psi^{\prime}\left(  t\right)
\right\vert ^{2}\mathrm{d}t\right\vert ^{\frac{1}{2}}\right\} \\
&  \leq & \left\vert 1-r^{-1}\right\vert ^{\frac{1}{2}}\left\{  \left\Vert
r^{1/2}\psi^{\prime\prime}\right\Vert _{L^{2}\left(  0,\infty\right)
}+\left(  1/2\right)  \left\Vert r^{-1/2}\psi^{\prime}\right\Vert
_{L^{2}\left(  0,\infty\right)  }\right\}  .
\end{eqnarray*}

(ii) For the first inequality, we employ the Leibniz-Newton formula%
\[
r^{-1}\psi\left(  r\right)  -\psi\left(  1\right)  =\int_{1}^{r}\left[
t^{-1}\psi\left(  t\right)  \right]  ^{\prime}\mathrm{d}t,
\]
together with the estimate%
\begin{eqnarray*}
\lefteqn{  \left\vert \int_{1}^{r}t^{-1/2}\left(  t^{1/2}\left[  t^{-1}\psi\left(
t\right)  \right]  ^{\prime}\right)  \mathrm{d}t\right\vert  } \\
&  \leq & \left\vert \int_{1}^{r}t^{-1}\mathrm{d}t\right\vert ^{\frac{1}{2}%
}\left\vert \int_{1}^{r}\left\vert t^{1/2}\left[  t^{-1}\psi\left(  t\right)
\right]  ^{\prime}\right\vert ^{2}\mathrm{d}t\right\vert ^{\frac{1}{2}}\\
&  \leq & \left\vert \ln r\right\vert ^{\frac{1}{2}}\left\Vert r^{1/2}\left(
r^{-1}\psi\right)  ^{\prime}\right\Vert _{L^{2}\left(  0,\infty\right)  },
\end{eqnarray*}
the last norm being finite due to $r^{1/2}\left(  r^{-1}\psi\right)  ^{\prime
}=r^{-1/2}\psi^{\prime}-r^{-3/2}\psi$. Since $\psi^{\prime}\in AC_{\mathrm{loc}%
}\left(  0,\infty\right)  $, the second inequality is obtained via%
\[
\psi^{\prime}\left(  r\right)  -\psi^{\prime}\left(  1\right)  =\int_{1}%
^{r}\psi^{\prime\prime}\left(  t\right)  \mathrm{d}t,
\]
followed by a similar estimate, this time in terms of $\left\Vert r^{1/2}%
\psi^{\prime\prime}\right\Vert _{L^{2}\left(  0,\infty\right)  }$.
\qed
\end{proof}

\subsection{Beppo Levi $L_{k}$-splines}

As observed in \cite{IMA13}, due to the Plancherel-type formula
(\ref{eq:Plancherel}), to obtain the variational characterization of the Beppo
Levi polyspline $s$ as minimizer of the polar thin plate energy
(\ref{eq:Du-polar}), it is sufficient to show that, for each $k\in\mathbb{Z}$,
the amplitude coefficient $\widehat{s}_{k}$ minimizes the corresponding energy
component (\ref{eq:k-norm}). Letting $Q\left(  r,g,g^{\prime},g^{\prime\prime
}\right)  $ denote the integrand of (\ref{eq:k-norm}), classical calculus of
variations considerations imply that, except at the interpolation locations
$r_{1}$,\ldots, $r_{n}$, a minimizer of (\ref{eq:k-norm}) should satisfy the
Euler-Lagrange equation%
\[
\partial_{g}Q-\frac{\mathrm{d}}{\mathrm{d}r}\partial_{g^{\prime}}%
Q+\frac{\mathrm{d}^{2}}{\mathrm{d}r^{2}}\partial_{g^{\prime\prime}}Q=0.
\]
The resulting left-hand side Euler-Lagrange differential operator is given, up
to a constant factor, by%
\begin{eqnarray*}
L_{k}  &  := & r\frac{\mathrm{d}^{4}}{\mathrm{d}r^{4}}+2\frac{\mathrm{d}^{3}%
}{\mathrm{d}r^{3}}-\frac{2k^{2}+1}{r}\frac{\mathrm{d}^{2}}{\mathrm{d}r^{2}%
}+\frac{2k^{2}+1}{r^{2}}\frac{\mathrm{d}}{\mathrm{d}r}+\frac{k^{4}-4k^{2}%
}{r^{3}}\\
&  = & r\left(  \frac{\mathrm{d}^{2}}{\mathrm{d}r^{2}}+\frac{1}{r}\frac
{\mathrm{d}}{\mathrm{d}r}-\frac{k^{2}}{r^{2}}\right)  ^{2}.
\end{eqnarray*}
Therefore, $\widehat{s}_{k}$ should necessarily be annihilated by $L_{k}$ on
each subinterval $\left(  0,r_{1}\right)  $, $\left(  r_{1},r_{2}\right)  $,
\ldots, $\left(  r_{n},\infty\right)  $.

The null-space $\mathrm{Ker}L_{k}$ is computed in \cite{Ku01} via the
substitution $r=e^{t}$, $\frac{\mathrm{d}}{\mathrm{d}t}=r\frac{\mathrm{d}%
}{\mathrm{d}r}$, which transforms $L_{k}$ into a differential operator with
constant coefficients in variable $t$. Standard factorization then implies%
\begin{equation}
L_{k}=\frac{1}{r^{3}}\left(  r\frac{\mathrm{d}}{\mathrm{d}r}-\left\vert
k\right\vert \right)  \left(  r\frac{\mathrm{d}}{\mathrm{d}r}-\left\vert
k\right\vert -2\right)  \left(  r\frac{\mathrm{d}}{\mathrm{d}r}+\left\vert
k\right\vert \right)  \left(  r\frac{\mathrm{d}}{\mathrm{d}r}+\left\vert
k\right\vert -2\right)  , \label{eq:Lk-fact1}%
\end{equation}
hence%
\begin{equation}
\mathrm{Ker}L_{k}=\left\{
\begin{array}
[c]{ll}%
\mathrm{span}\left\{  r^{2},r^{2}\ln r,1,\ln r\right\}  , & \mathrm{if\ 
}k=0\mathrm{,}\\
\mathrm{span}\left\{  r^{3},r,r\ln r,r^{-1}\right\}  , & \mathrm{if\ }\left\vert
k\right\vert =1\mathrm{,}\\
\mathrm{span}\left\{  r^{\left\vert k\right\vert +2},r^{\left\vert
k\right\vert },r^{-\left\vert k\right\vert +2},r^{-\left\vert k\right\vert
}\right\}  , & \mathrm{if\ }\left\vert k\right\vert \geq2\mathrm{.}%
\end{array}
\right.  \label{eq:ker}%
\end{equation}
Moreover, the condition that the polar Beppo Levi energy component
(\ref{eq:k-norm}) is finite further restricts the form of $\widehat{s}_{k}$ on
the extreme intervals $\left(  0,r_{1}\right)  $ and $\left(  r_{n}%
,\infty\right)  $. Specifically, for $k\not =0$, evaluating (\ref{eq:k-norm})
for each of the four generating functions of $\mathrm{Ker}L_{k}$, we obtain
the necessary conditions%
\[
\widehat{s}_{k}\left(  r\right)  \in\left\{
\begin{array}
[c]{ll}%
\mathrm{span}\left\{  r^{\left\vert k\right\vert +2},r^{\left\vert
k\right\vert }\right\}  , & \mathrm{for\ }r\in\left(  0,r_{1}\right)  \mathrm{,}\\
\mathrm{span}\left\{  r^{-\left\vert k\right\vert +2},r^{-\left\vert
k\right\vert }\right\}  , & \mathrm{for\ }r\in\left(  r_{n},\infty\right)
\mathrm{.}%
\end{array}
\right.
\]
Note that $\mathrm{span}\left\{  r^{\left\vert k\right\vert +2},r^{\left\vert
k\right\vert }\right\}  =\mathrm{Ker}G_{k}$ and $\mathrm{span}\left\{
r^{-\left\vert k\right\vert +2},r^{-\left\vert k\right\vert }\right\}
=\mathrm{Ker}R_{k}$, where%
\begin{eqnarray}
G_{k}  &  := & \frac{1}{r}\left[  \frac{\mathrm{d}^{2}}{\mathrm{d}r^{2}}%
-\frac{2\left\vert k\right\vert +1}{r}\frac{\mathrm{d}}{\mathrm{d}r}%
+\frac{\left\vert k\right\vert \left(  \left\vert k\right\vert +2\right)
}{r^{2}}\right] \nonumber\\
&  = & \frac{1}{r^{3}}\left(  r\frac{\mathrm{d}}{\mathrm{d}r}-\left\vert
k\right\vert \right)  \left(  r\frac{\mathrm{d}}{\mathrm{d}r}-\left\vert
k\right\vert -2\right)  ,\nonumber\\
R_{k}  &  := & \frac{1}{r}\left[  \frac{\mathrm{d}^{2}}{\mathrm{d}r^{2}}%
+\frac{2\left\vert k\right\vert -1}{r}\frac{\mathrm{d}}{\mathrm{d}r}%
+\frac{\left\vert k\right\vert \left(  \left\vert k\right\vert -2\right)
}{r^{2}}\right] \nonumber\\
&  = & \frac{1}{r^{3}}\left(  r\frac{\mathrm{d}}{\mathrm{d}r}+\left\vert
k\right\vert \right)  \left(  r\frac{\mathrm{d}}{\mathrm{d}r}+\left\vert
k\right\vert -2\right)  . \label{eq:end-ops}%
\end{eqnarray}

\begin{remark}
It can be verified that $r^{-3}$ is the only factor of the form
$r^{\alpha}$ which, when inserted in front of the last two brackets in
the right-hand side of the above formulae, turns $G_{k}$ and $R_{k}%
$ into mutually adjoint operators. Indeed, the formal adjoint of %
$G_{k}$ is%
\[
G_{k}^{\ast}=\frac{\mathrm{d}^{2}}{\mathrm{d}r^{2}}\left(  \frac{1}{r}%
\cdot\right)  +\left(  2\left\vert k\right\vert +1\right)  \frac{\mathrm{d}%
}{\mathrm{d}r}\left(  \frac{1}{r^{2}}\cdot\right)  +\frac{\left\vert
k\right\vert \left(  \left\vert k\right\vert +2\right)  }{r^{3}}=R_{k}%
\]
and a similar computation shows $R_{k}^{\ast}=G_{k}$.
\end{remark}

Recall the notation $\rho:=\left\{  r_{1},\ldots,r_{n}\right\}  $ used in the Introduction.

\begin{definition}
\label{def:natLspline}Let $k\not =0$. A function $\eta:\left[  0,\infty
\right)  \rightarrow\mathbb{C}$ is called a \emph{Beppo Levi }$L_{k}%
$\emph{-spline on }$\rho$ if the following conditions hold:

\emph{(i)} $L_{k}\eta\left(  r\right)  =0,\ \forall r\in\left(
r_{j},r_{j+1}\right),\ \forall j\in\left\{  1,\ldots,n-1\right\}  $;

\emph{(ii)} $G_{k}\eta\left(  r\right)  =0$, $\forall r\in\left(
0,r_{1}\right)  $, and $R_{k}\eta\left(  r\right)  =0$, $\forall r>r_{n}$.

\emph{(iii)} $\eta$ is $C^{2}$-continuous at each knot $r_{1}$,\ldots, $r_{n}$.

\noindent The space of all Beppo Levi $L_{k}$-splines on $\rho$ will be
labelled $\mathcal{S}_{k}\left(  \rho\right)  $.
\end{definition}

\noindent Due to conditions (ii), $\mathcal{S}_{k}\left(  \rho\right)  $ is a
subspace of $\Lambda_{1}$ if $\left\vert k\right\vert =1$, and of $\Lambda
_{2}$ if $\left\vert k\right\vert \geq2$.

For $k=0$, the related notion of a Beppo Levi $L_{0}$-spline is treated in
\cite{RTPS}. In this case, the correct left/right operators $G_{0}$ and
$R_{0}$ on the extreme intervals are not obtained by just letting $k=0$ in
(\ref{eq:end-ops}). Also, $G_{0}$ and $R_{0}$ are not anymore mutually adjoint.

The proof of the next result follows from the definition of biharmonic Beppo
Levi polysplines on annuli \cite{MMA12}.

\begin{proposition}
A univariate function $\eta:\left[  0,\infty\right)  \rightarrow\mathbb{C}$ is
a Beppo Levi $L_{k}$-spline on $\rho$, i.e. $\eta\in\mathcal{S}_{k}\left(
\rho\right)  $, if and only if the polar surface $s\left(  r,\theta\right)
:=\eta\left(  r\right)  e^{-ik\theta}$ is a biharmonic Beppo Levi polyspline
on annuli determined by $\rho$.
\end{proposition}

We now review some relevant literature. It was pointed out by Kounchev
\cite[p.\ 91]{Ku01} that, on any interval of positive real numbers, the null
space of $L_{k}$ can be described as an extended complete Chebyshev (ECT)
system in the sense of Karlin and Ziegler \cite{KZ66}, via the representation%
\[
r^{\left\vert k\right\vert }L_{k}=D_{4}D_{3}D_{2}D_{1}=\frac{\mathrm{d}%
}{\mathrm{d}r}r^{2\left\vert k\right\vert +1}\frac{\mathrm{d}}{\mathrm{d}%
r}\frac{1}{r^{2\left\vert k\right\vert +1}}\frac{\mathrm{d}}{\mathrm{d}%
r}r^{2\left\vert k\right\vert +1}\frac{\mathrm{d}}{\mathrm{d}r}\frac
{1}{r^{\left\vert k\right\vert }},
\]
where $D_{1}=\frac{\mathrm{d}}{\mathrm{d}r}\left(  \frac{1}{r^{\left\vert
k\right\vert }}\cdot\right)  $, $D_{2}=D_{4}=\frac{\mathrm{d}}{\mathrm{d}%
r}\left(  r^{2\left\vert k\right\vert +1}\cdot\right)  $, $D_{3}%
=\frac{\mathrm{d}}{\mathrm{d}r}\left(  \frac{1}{r^{2\left\vert k\right\vert
+1}}\cdot\right)  $. We also observe, following Schumaker \cite[p.\ 398]%
{LLS07}, that $L_{k}$ possesses the factorization%
\begin{equation}
L_{k}=M_{k}^{\ast}M_{k}, \label{eq:Lk-fact2}%
\end{equation}
where $M_{k}^{\ast}$ denotes the formal adjoint of
\begin{eqnarray*}
M_{k}  &  := & \frac{1}{\sqrt{r^{2\left\vert k\right\vert +1}}}\frac{\mathrm{d}%
}{\mathrm{d}r}r^{2\left\vert k\right\vert +1}\frac{\mathrm{d}}{\mathrm{d}%
r}\frac{1}{r^{\left\vert k\right\vert }}=\sqrt{r}\left(  \frac{\mathrm{d}^{2}%
}{\mathrm{d}r^{2}}+\frac{1}{r}\frac{\mathrm{d}}{\mathrm{d}r}-\frac{k^{2}%
}{r^{2}}\right) \\
&  = & r^{-3/2}\left(  r\frac{\mathrm{d}}{\mathrm{d}r}-\left\vert k\right\vert
\right)  \left(  r\frac{\mathrm{d}}{\mathrm{d}r}+\left\vert k\right\vert
\right)  .
\end{eqnarray*}
Due to (\ref{eq:Lk-fact2}), a function that satisfies conditions (i) and (iii)
of Definition~\ref{def:natLspline} can be characterized as a `generalized spline'
or `$M_{k}$-spline' on $\left[  r_{1},r_{n}\right]  $ in the sense of Ahlberg,
Nilson, and Walsh \cite{ANW67}, Schultz and Varga \cite{SV67}. However, our
labeling such a function as a `$L_{k}$-spline' agrees with the terminology of
Lucas \cite{L70} and Jerome and Pierce \cite{JP72}, which is more adequate, in
view of the fact that $L_{k}$ may possess other factorizations of the type
(\ref{eq:Lk-fact2}). Indeed, for $k\not =0$, our adjoint boundary operators
$G_{k}$ and $R_{k}$ actually generate, via (\ref{eq:Lk-fact1}), the
factorization%
\begin{equation}
L_{k}=G_{k}r^{3}R_{k}=\widetilde{L}_{k}^{\ast}\widetilde{L}_{k},\quad
\mathrm{where\ }\widetilde{L}_{k}:=r^{3/2}R_{k}. \label{eq:Lk-fact3}%
\end{equation}
This differs from (\ref{eq:Lk-fact2}) for $\left\vert k\right\vert \geq2$,
while it coincides with (\ref{eq:Lk-fact2}) for $\left\vert k\right\vert =1$.

On the other hand, the `natural' end conditions of $L$-spline literature (see
\cite{LLS07}) are always formulated in terms of a single differential operator
at both ends of the interpolation domain. It is thus remarkable that adjoint
boundary operators as in condition (ii) of our definition have also occured in
\cite{CA11}, in the context of exponential $L$-splines generated as Fourier
coefficients of Beppo Levi polyspline surfaces on parallel strips. Such
exponential $L$-splines coincide in fact with Mat\'{e}rn kernels on the full
real line (for Mat\'{e}rn kernels on a compact interval, see \cite{CFM14}). As
shown in \cite{JAT08}, adjoint $L$-spline end conditions are intimately
connected to Wiener-Hopf factorizations for semi-cardinal interpolation.

\section{Interpolation with Beppo Levi $L_{k}$-splines}

\subsection{A fundamental identity}

We employ the notations introduced in the previous section.

\begin{theorem}
\label{thm:FI}\emph{(i)} Let $k\in\mathbb{Z}$, $\left\vert k\right\vert \geq
2$, and an arbitrary Beppo Levi $L_{k}$-spline $\eta\in\mathcal{S}_{k}\left(
\rho\right)  $. Also, assume that $\psi\in\Lambda_{2}$ vanishes on the
knot-set $\rho$:%
\begin{equation}
\psi\left(  r_{j}\right)  =0,\quad\forall j\in\left\{  1,\ldots,n\right\}  .
\label{eq:zero-data}%
\end{equation}
Then the following orthogonality relation holds:%
\begin{equation}
\int_{0}^{\infty}r^{3}\left[  R_{k}\eta\left(  r\right)  \right]  \left[
R_{k}\overline{\psi}\left(  r\right)  \right]  \mathrm{d}r=0. \label{eq:ortho}%
\end{equation}

\emph{(ii)} The same conclusion holds if $k=\pm1$ and $\psi\in\Lambda_{1}$
satisfies \emph{(\ref{eq:zero-data})}.
\end{theorem}

\begin{proof}
For convenience, let us denote the left-hand side of (\ref{eq:ortho}) by
$I_{k}:=I_{k}\left(  \eta,\psi\right)  $. Note that, for any $k\not =0$,
$\eta\in\mathcal{S}_{k}\left(  \rho\right)  $ implies $R_{k}\eta\left(
r\right)  =0$, $\forall r>r_{n}$, hence we can work with integral $I_{k}$ on
the integration domain $(0,r_{n}]$. Since%
\[
r^{3/2}R_{k}\psi=r^{1/2}\psi^{\prime\prime}+\left(  2\left\vert k\right\vert
-1\right)  r^{-1/2}\psi^{\prime}+\left\vert k\right\vert \left(  \left\vert
k\right\vert -2\right)  r^{-3/2}\psi,
\]
the hypotheses imply, via Cauchy-Schwarz inequality, that $I_{k}$ is an
absolutely convergent integral. Using the factorization of the operator
$R_{k}$ and making the notation%
\[%
\begin{array}
[c]{l}%
\eta_{1}\left(  r\right)  :=\left(  r\frac{\mathrm{d}}{\mathrm{d}%
r}+\left\vert k\right\vert -2\right)  \eta\left(  r\right)  ,\\
\psi_{1}\left(  r\right)  :=\left(  r\frac{\mathrm{d}}{\mathrm{d}%
r}+\left\vert k\right\vert -2\right)  \psi\left(  r\right)  ,
\end{array}
\]
we have%
\begin{eqnarray*}
I_{k}  &  = & \int_{0}^{r_{n}}r^{-3}\left[  \left(  r\frac{\mathrm{d}}%
{\mathrm{d}r}+\left\vert k\right\vert \right)  \eta_{1}\left(  r\right)
\right]  \left[  \left(  r\frac{\mathrm{d}}{\mathrm{d}r}+\left\vert
k\right\vert \right)  \overline{\psi_{1}}\left(  r\right)  \right]
\mathrm{d}r\\
&  = & \sum_{j=1}^{n}\int_{r_{j-1}}^{r_{j}}r^{-2}\left[  \left(  r\frac
{\mathrm{d}}{\mathrm{d}r}+\left\vert k\right\vert \right)  \eta_{1}\left(
r\right)  \right]  \frac{\mathrm{d}}{\mathrm{d}r}\overline{\psi_{1}}\left(
r\right)  \mathrm{d}r\\
&  & \mbox{}+\sum_{j=1}^{n}\int_{r_{j-1}}^{r_{j}}r^{-3}\left[  \left(  r\frac
{\mathrm{d}}{\mathrm{d}r}+\left\vert k\right\vert \right)  \eta_{1}\left(
r\right)  \right]  \left\vert k\right\vert \overline{\psi_{1}}\left(
r\right)  \mathrm{d}r,
\end{eqnarray*}
where all integrals remain absolutely convergent and $r_{0}:=0$. Next, we
apply integration by parts in each term of the first sum, which is permitted
due to the fact that $\psi_{1}\in AC_{\mathrm{loc}}\left(  0,\infty\right)  $.
Since
\[
\frac{\mathrm{d}}{\mathrm{d}r}\left\{  r^{-2}\left[  \left(  r\frac
{\mathrm{d}}{\mathrm{d}r}+\left\vert k\right\vert \right)  \eta_{1}\left(
r\right)  \right]  \right\}  =r^{-3}\left[  \left(  r\frac{\mathrm{d}%
}{\mathrm{d}r}-2\right)  \left(  r\frac{\mathrm{d}}{\mathrm{d}r}+\left\vert
k\right\vert \right)  \eta_{1}\left(  r\right)  \right]  ,
\]
we obtain%
\begin{eqnarray*}
I_{k}  &  = & \sum_{j=1}^{n}\left[  \overline{\psi_{1}}\left(  r\right)
r^{-2}\left(  r\frac{\mathrm{d}}{\mathrm{d}r}+\left\vert k\right\vert \right)
\eta_{1}\left(  r\right)  \right]  _{r_{j-1}}^{r_{j}}\\
&  & \mbox{}-\sum_{j=1}^{n}\int_{r_{j-1}}^{r_{j}}r^{-3}\overline{\psi_{1}}\left(
r\right)  \left(  r\frac{\mathrm{d}}{\mathrm{d}r}-\left\vert k\right\vert
-2\right)  \left(  r\frac{\mathrm{d}}{\mathrm{d}r}+\left\vert k\right\vert
\right)  \eta_{1}\left(  r\right)  \mathrm{d}r.
\end{eqnarray*}
Since $\psi$ has continuity $C^{1}$ and $\eta$ has continuity $C^{2}$, the
first sum of the last display is telescopic, hence we only have to evaluate
the boundary terms corresponding to $r:=r_{n}$ and $r:=r_{0}=0$. Note that the
boundary term at $r_{n}$ is zero, since the condition $R_{k}\eta\left(
r\right)  =0$, $\forall r>r_{n}$, of a Beppo Levi $L_{k}$-spline implies, by
continuity, the relation $\left[  \left(  r\frac{\mathrm{d}}{\mathrm{d}%
r}+\left\vert k\right\vert \right)  \eta_{1}\left(  r\right)  \right]
_{r=r_{n}}=0$.

For the boundary term at $0$, consider first the case $\left\vert k\right\vert
\geq2$. Then the left end condition $G_{k}\eta\left(  r\right)  =0$,
\emph{i.e.} $\eta\in\mathrm{span}\left\{  r^{\left\vert k\right\vert
+2},r^{\left\vert k\right\vert }\right\}  $, for $r\in\left(  0,r_{1}\right)
$, implies%
\[
r^{-2}\left[  \left(  r\frac{\mathrm{d}}{\mathrm{d}r}+\left\vert k\right\vert
\right)  \eta_{1}\left(  r\right)  \right]  =O\left(  r^{\left\vert
k\right\vert -2}\right)  ,\quad\mathrm{as\ }r\rightarrow0.
\]
Since, by Lemma~\ref{le:aux}, $\psi_{1}\left(  r\right)  =O\left(  r\right)
$, as $r\rightarrow0$, we deduce that the boundary term at $0$ vanishes if
$\left\vert k\right\vert \geq2$. If $\left\vert k\right\vert =1$, the left end
condition implies $\eta\in\mathrm{span}\left\{  r^{3},r^{1}\right\}  $, for
$r\in\left(  0,r_{1}\right)  $, hence%
\[
r^{-2}\left[  \left(  r\frac{\mathrm{d}}{\mathrm{d}r}+1\right)  \eta
_{1}\left(  r\right)  \right]  =cr,\quad\forall r\in\left(  0,r_{1}\right)  ,
\]
for some constant $c$. Since, by Lemma~\ref{le:aux}, in this case $\psi
_{1}\left(  r\right)  =O\left(  r\left\vert \ln r\right\vert ^{1/2}\right)  $,
as $r\rightarrow0$, it follows that the boundary term at $0$ also vanishes if
$\left\vert k\right\vert =1$.

On the other hand, for each $j\in\left\{  1,\ldots,n\right\}  $, since
$\eta\in\mathrm{Ker}L_{k}$ on the interval $\left(  r_{j-1},r_{j}\right)  $,
there exists a constant $c_{j}$ such that%
\[
\left(  r\frac{\mathrm{d}}{\mathrm{d}r}-\left\vert k\right\vert -2\right)
\left(  r\frac{\mathrm{d}}{\mathrm{d}r}+\left\vert k\right\vert \right)
\eta_{1}\left(  r\right)  =c_{j}r^{\left\vert k\right\vert },\quad\forall
r\in\left(  r_{j-1},r_{j}\right)  .
\]
Hence%
\begin{eqnarray*}
I_{k}  &  = & \sum_{j=1}^{n}c_{j}\int_{r_{j-1}}^{r_{j}}r^{\left\vert k\right\vert
-3}\left(  r\frac{\mathrm{d}}{\mathrm{d}r}+\left\vert k\right\vert -2\right)
\overline{\psi}\left(  r\right)  \mathrm{d}r\\
&  = & \sum_{j=1}^{n}c_{j}\int_{r_{j-1}}^{r_{j}}\frac{\mathrm{d}}{\mathrm{d}%
r}\left[  r^{\left\vert k\right\vert -2}\overline{\psi}\left(  r\right)
\right]  \mathrm{d}r=\sum_{j=1}^{n}c_{j}\left[  r^{\left\vert k\right\vert
-2}\overline{\psi}\left(  r\right)  \right]  _{r_{j-1}}^{r_{j}}.
\end{eqnarray*}
For $\left\vert k\right\vert \geq2$, since Lemma~\ref{le:aux} implies
$r^{\left\vert k\right\vert -2}\overline{\psi}\left(  r\right)  =O\left(
r\right)  $, as $r\rightarrow0$, and, by hypothesis, $\psi\left(
r_{j}\right)  =0$, $\forall j\in\left\{  1,\ldots,n\right\}  $, we deduce
$I_{k}=0$, as stated. For $\left\vert k\right\vert =1$, we reach the same
conclusion without the need to investigate $r^{-1}\overline{\psi}\left(
r\right)  $ as $r\rightarrow0$, since in this case $c_{1}=0$.
\qed
\end{proof}

\subsection{Existence, uniqueness, and optimality}

\begin{theorem}
\label{thm:EU}Let $\nu_{1},\ldots,\nu_{n}$ be arbitrary real values, where
$n\geq2$. For each $k\not =0$, there exists a unique Beppo Levi $L_{k}$-spline
$\sigma\in\mathcal{S}_{k}\left(  \rho\right)  $, such that%
\begin{equation}
\sigma\left(  r_{j}\right)  =\nu_{j},\quad j\in\left\{  1,\ldots,n\right\}  .
\label{eq:int-cond}%
\end{equation}

\end{theorem}

\begin{proof}
It is sufficient to prove the existence of a unique function $\widetilde
{\sigma}\in C^{2}\left[  r_{1},r_{n}\right]  $ such that $\widetilde{\sigma
}\in\mathrm{Ker}L_{k}$ on each subinterval $\left(  r_{j-1},r_{j}\right)  $
with $j\in\left\{  2,\ldots,n\right\}  $, $\widetilde{\sigma}$ satisfies the
interpolation conditions (\ref{eq:int-cond}) in place of $\sigma$, and the
following endpoint conditions hold:%
\begin{equation}%
\begin{array}
[c]{l}%
\left[  \left(  r\frac{\mathrm{d}}{\mathrm{d}r}-\left\vert k\right\vert
\right)  \left(  r\frac{\mathrm{d}}{\mathrm{d}r}-\left\vert k\right\vert
-2\right)  \widetilde{\sigma}\left(  r\right)  \right]  _{r\rightarrow
r_{1}^{+}}=0,\\
\left[  \left(  r\frac{\mathrm{d}}{\mathrm{d}r}+\left\vert k\right\vert
\right)  \left(  r\frac{\mathrm{d}}{\mathrm{d}r}+\left\vert k\right\vert
-2\right)  \widetilde{\sigma}\left(  r\right)  \right]  _{r\rightarrow
r_{n}^{-}}=0.
\end{array}
\label{eq:endpt}%
\end{equation}
Indeed, such a function $\widetilde{\sigma}$ can be uniquely extended to the
required Beppo Levi $L_{k}$-spline $\sigma\in\mathcal{S}_{k}\left(
\rho\right)  $ by defining%
\[
\sigma\left(  r\right)  :=\left\{
\begin{array}
[c]{ll}%
c_{1}r^{\left\vert k\right\vert +2}+c_{2}r^{\left\vert k\right\vert }, &
\mathrm{if\ }0<r<r_{1},\\
& \\
\widetilde{\sigma}\left(  r\right)  , & \mathrm{if\ }r_{1}\leq r\leq r_{n},\\
& \\
c_{3}r^{-\left\vert k\right\vert +2}+c_{4}r^{-\left\vert k\right\vert }, &
\mathrm{if\ }r_{n}<r.
\end{array}
\right.
\]
To verify this, note that $c_{1}$ and $c_{2}$ (respectively, $c_{3}$ and
$c_{4}$) are uniquely determined by the conditions that $\sigma$ and
$\sigma^{\prime}$ are continuous at $r_{1}$ (respectively, at $r_{n}$). The
continuity of $\sigma^{\prime\prime}$ at $r_{1}$ and $r_{n}$ then follows
automatically from (\ref{eq:endpt}) and from the properties $G_{k}%
\sigma\left(  r\right)  =0$, $\forall r\in\left(  0,r_{1}\right)  $, and
$R_{k}\sigma\left(  r\right)  =0$, $\forall r>r_{n}$.

Now, a function $\widetilde{\sigma}$ with the properties stated in the
previous paragraph is determined by four coefficients on each of the $n-1$
subintervals $\left(  r_{j-1},r_{j}\right)  $, $j\in\left\{  2,\ldots
,n\right\}  $. These coefficients are coupled by three $C^{2}$-continuity
conditions at each interior knot $r_{2},\ldots,r_{n-1}$, the endpoint
conditions (\ref{eq:endpt}), and the $n$ interpolation conditions
(\ref{eq:int-cond}), which amount to a $4\left(  n-1\right)  \times4\left(
n-1\right)  $ system of linear equations.

To show that this system has a unique solution, we assume zero interpolation
data: $\nu_{j}=0$, $j\in\left\{  1,\ldots,n\right\}  $. Then the system
becomes homogeneous, since the endpoint conditions and the continuity
conditions at the interior knots were already homogeneous linear equations.
Let $\widetilde{\sigma}$ be determined by an arbitrary solution of this
homogeneous system and let $\sigma\in\mathcal{S}_{k}\left(  \rho\right)  $ be
the unique extension of $\widetilde{\sigma}$ to a Beppo Levi $L_{k}$-spline.
Taking $\eta=\psi:=\sigma$ in (\ref{eq:ortho}), we obtain $R_{k}\sigma\left(
r\right)  =0$, \emph{i.e.} $\sigma\in\mathrm{span}\left\{  r^{-\left\vert
k\right\vert +2},r^{-\left\vert k\right\vert }\right\}  $, for $r\in\left(
0,\infty\right)  $. Since $\sigma\left(  r_{j}\right)  =0$, $j\in\left\{
1,\ldots,n\right\}  $, and $n\geq2$, we deduce $\sigma\equiv0$. Therefore the
above homogeneous system admits only the trivial solution, which concludes the proof.
\qed
\end{proof}

Theorem~\ref{thm:EU} also extends to the case $n=1$. Indeed, for each integer
$k\not =0$, it is straightforward to verify that there exists a unique
function $\varphi_{k}$ with the properties: $G_{k}\varphi_{k}\left(  r\right)
=0$ for $0<r<1$, $R_{k}\varphi_{k}\left(  r\right)  =0$ for $r>1$,
$\varphi_{k}$ is $C^{2}$-continuous at $r=1$, and $\varphi_{k}\left(
1\right)  =1$. Its expression%
\begin{equation}
\varphi_{k}\left(  r\right)  =\frac{1}{2}\left\{
\begin{array}
[c]{l}%
r^{\left\vert k\right\vert }\left[  \left(  1+\left\vert k\right\vert \right)
+\left(  1-\left\vert k\right\vert \right)  r^{2}\right]  ,\quad0\leq
r\leq1,\\
\\
r^{-\left\vert k\right\vert }\left[  \left(  1-\left\vert k\right\vert
\right)  +\left(  1+\left\vert k\right\vert \right)  r^{2}\right]  ,\quad1<r,
\end{array}
\right.  \label{eq:basis-fn}%
\end{equation}
was given in \cite[(3.10)]{MMA12} for $\left\vert k\right\vert \geq2$ and is
also seen to hold for $\left\vert k\right\vert =1$. Hence, if $\rho=\left\{
r_{1}\right\}  $, then $\sigma:=\nu_{1}\varphi_{k}\left(  \cdot/r_{1}\right)
$ is the unique Beppo Levi $L_{k}$-spline in $S_{k}\left(  \rho\right)  $,
such that $\sigma\left(  r_{1}\right)  =\nu_{1}$. As shown by the next result,
if $\left\vert k\right\vert \geq2$ and $n\geq2$, the dilates of $\varphi_{k}$
also provide a basis for a linear representation of the interpolant of
Theorem~\ref{thm:EU}.

\begin{theorem}
Assume that $\left\vert k\right\vert \geq2$, $n\geq2$, and let $\sigma$ be the
Beppo Levi $L_{k}$-spline satisfying the interpolation conditions
\emph{(\ref{eq:int-cond})} of Theorem~\ref{thm:EU} for given values $\nu_{1}$,
\ldots, $\nu_{n}$ at the knot-set $\rho$. Then there exist unique coefficients
$a_{1}$, \ldots, $a_{n}$, such that%
\begin{equation}
\sigma\left(  r\right)  =\sum\limits_{j=1}^{n}a_{j}\varphi_{k}\left(  \frac
{r}{r_{j}}\right)  ,\quad\forall r\geq0. \label{eq:rep}%
\end{equation}

\end{theorem}

\noindent This result was established in \cite[Lemma 3]{MMA12} for the
special case in which $\sigma$ satisfies Lagrange interpolation conditions.
The proof given there also applies to our general interpolation conditions
(\ref{eq:int-cond}). Note that representation (\ref{eq:rep}) does not hold for
$\left\vert k\right\vert =1$, but a similar representation for $k=0$ appears
in \cite[Theorem~4]{RTPS}.

\begin{remark}
For $\left\vert k\right\vert \geq2$, it was observed in \cite{MMA12} that, if we make 
the notation
\[
\psi _{k}\left( t\right) :=e^{-t}\varphi _{k}\left( e^{t}\right) =\frac{1}{2}%
e^{-\left\vert k\right\vert \left\vert t\right\vert }\left[ \left(
1-\left\vert k\right\vert \right) e^{-\left\vert t\right\vert }+\left(
1+\left\vert k\right\vert \right) e^{\left\vert t\right\vert }\right] ,\quad
t\in \mathbb{R},
\]
then $\psi _{k}$ is a positive definite function, due to its positive Fourier transform
\[
\widehat{\psi }_{k}\left( \tau \right) =\int_{-\infty }^{\infty }e^{-it\tau
}\psi _{k}\left( t\right) dt=\frac{4\left\vert k\right\vert \left(
k^{2}-1\right) }{\left[ \left( \left\vert k\right\vert -1\right) ^{2}+\tau
^{2}\right] \left[ \left( \left\vert k\right\vert +1\right) ^{2}+\tau ^{2}%
\right] },\quad \tau \in \mathbb{R}.
\]
This formula shows that $\psi _{k}$ belongs to 
the class of generalized Whittle-Mat\'{e}rn-Sobolev kernels recently studied by 
Bozzini, Rossini, and Schaback \cite{BRS13}.
\end{remark}

The next result shows that our Beppo Levi $L_{k}$-spline interpolants minimize
the functional (\ref{eq:k-norm}), subject to the interpolation conditions.

\begin{theorem}
\label{thm:var-char}Given $k\not =0$ and arbitrary real values $\nu_{1}$,
$\nu_{2}$,\ldots, $\nu_{n}$, let $\sigma$ denote the unique Beppo Levi $L_{k}%
$-spline obtained in Theorem~\ref{thm:EU}. Then $\left\Vert \sigma\right\Vert
_{k}<\left\Vert g\right\Vert _{k}$ whenever $g$ satisfies the same
interpolation conditions \emph{(\ref{eq:int-cond})} as $\sigma$ and
$g\not =\sigma$, where $g\in\Lambda_{1}$ if $\left\vert k\right\vert =1$,
while $g\in\Lambda_{2}$ if $\left\vert k\right\vert \geq2$.
\end{theorem}

\begin{proof}
Letting $\eta:=\sigma$, $\psi:=g-\sigma$, the hypotheses imply that $\psi$
satisfies (\ref{eq:zero-data}), hence (\ref{eq:ortho}) holds by
Theorem~\ref{thm:FI}. Since $\psi^{\prime}\in AC_{\mathrm{loc}}\left(
0,\infty\right)  $ and $\psi\in\Lambda_{1}$ if $\left\vert k\right\vert =1$,
while $\psi\in\Lambda_{2}$ if $\left\vert k\right\vert \geq2$, we can use the
proof of \cite[Formula (5.3)]{MMA12} for $k\not =0$ to show that%
\[
\int_{0}^{\infty}r^{3}\left[  R_{k}\eta\left(  r\right)  \right]  \left[
R_{k}\overline{\psi}\left(  r\right)  \right]  \mathrm{d}r=\left\langle
\eta,\psi\right\rangle _{k},
\]
where%
\begin{eqnarray*}
\left\langle \eta,\psi\right\rangle _{k}  &  := & \int_{0}^{\infty}\left\{
\eta^{\prime\prime}\overline{\psi}^{\prime\prime}+2k^{2}\left[  \frac{\eta
}{r^{2}}-\frac{\eta^{\prime}}{r}\right]  \left[  \frac{\overline{\psi}}{r^{2}%
}-\frac{\overline{\psi}^{\prime}}{r}\right]  \right. \\
&  & \left.  +\left[  \frac{k^{2}\eta}{r^{2}}-\frac{\eta^{\prime}}{r}\right]
\left[  \frac{k^{2}\overline{\psi}}{r^{2}}-\frac{\overline{\psi}^{\prime}}%
{r}\right]  \right\}  r\,\mathrm{d}r.
\end{eqnarray*}
Therefore (\ref{eq:ortho}) implies the orthogonality property%
\[
\left\langle \sigma,g-\sigma\right\rangle _{k}=0,
\]
from which%
\begin{equation}
\left\Vert g\right\Vert _{k}^{2}=\left\Vert \sigma\right\Vert _{k}%
^{2}+\left\Vert g-\sigma\right\Vert _{k}^{2}, \label{eq:pythagoras}%
\end{equation}
and $\left\Vert g\right\Vert _{k}\geq\left\Vert \sigma\right\Vert _{k}$, with
equality only if $\left\Vert g-\sigma\right\Vert _{k}=0$. The last relation
implies $g\equiv\sigma$ if $\left\vert k\right\vert \geq2$, since $\left\Vert
\cdot\right\Vert _{k}$ is a norm in this case. If $\left\vert k\right\vert
=1$, the semi-norm $\left\Vert g-\sigma\right\Vert _{k}$ vanishes if and only
if $g\left(  r\right)  -\sigma\left(  r\right)  =ar$, $\forall r\in\left(
0,\infty\right)  $, for some constant $a$. Since $g-\sigma$ takes zero values
at the knots $r_{1}$,\ldots, $r_{n}$, we deduce again $g\equiv\sigma$, which
completes the proof.
\qed
\end{proof}

\section{Approximation orders}

For each $k\not =0$, the following result establishes $L^{\infty}$ and $L^{2}%
$-error bounds for interpolation with Beppo Levi $L_{k}$-splines to data
functions from $\Lambda_{1}$ or $\Lambda_{2}$.

\begin{theorem}
\label{thm:AO-Lspl}Let $\rho:=\left\{  r_{1},\ldots,r_{n}\right\}  $ be a set
of nodes with $0<r_{1}<\ldots<r_{n}$, $n\geq2$, and $h:=\max_{1\leq j\leq
n-1}\left(  r_{j+1}-r_{j}\right)  $. For an integer $k\not =0$, let $g:\left(
0,\infty\right)  \rightarrow\mathbb{R}$ be a data function such that
$g\in\Lambda_{1}$ if $\left\vert k\right\vert =1$, while $g\in\Lambda_{2}$ if
$\left\vert k\right\vert \geq2$. Let $\sigma\in\mathcal{S}_{k}\left(
\rho\right)  $ be the Beppo Levi $L_{k}$-spline of Theorem~\ref{thm:EU},
corresponding to data values $\nu_{j}:=g\left(  r_{j}\right)  $, $1\leq j\leq
n$. Then, for $m\in\left\{  0,1\right\}  $, we have the error bounds:
\begin{equation}
\left\Vert \frac{d^{m}}{dr^{m}}\left(  g-\sigma\right)  \right\Vert
_{L^{\infty}\left[  r_{1},r_{n}\right]  }\leq\frac{1}{2^{1-m}\sqrt{r_{1}}%
}h^{3/2-m}\left\Vert g\right\Vert _{k}, \label{eq:Linfty}%
\end{equation}%
\begin{equation}
\left\Vert \frac{d^{m}}{dr^{m}}\left(  g-\sigma\right)  \right\Vert
_{L^{2}\left[  r_{1},r_{n}\right]  }\leq\frac{1}{2^{1-m}\sqrt{r_{1}}}%
h^{2-m}\left\Vert g\right\Vert _{k}. \label{eq:L2}%
\end{equation}

\end{theorem}

\begin{proof}
Similar error bounds for $k=0$ were obtained in \cite[Theorems 5 \& 6]{RTPS},
along the lines of the classical error analysis for generalized splines
\cite{ANW67}. The same arguments are also seen to apply to the present case
$k\not =0$, by replacing the semi-norm $\left\Vert \cdot\right\Vert _{0}$ of
\cite{RTPS} with $\left\Vert \cdot\right\Vert _{k}$ and using the inequality
$\int_{0}^{\infty}r\left\vert g^{\prime\prime}\left(  r\right)  \right\vert
^{2}\mathrm{d}r\leq\left\Vert g\right\Vert _{k}^{2}$, valid for any data
function $g$ as in the hypothesis.
\qed
\end{proof}

\begin{remark}
As in \cite{RTPS}, the bounds (\ref{eq:Linfty}) and (\ref{eq:L2}) also imply an $L^{p}%
$-error bound for $p\in\left(  2,\infty\right)  $. Moreover, a similar
analysis to that of \cite[Theorem 7]{RTPS} shows that the exponents of $h$ in
the above error bounds cannot be increased for the classes $\Lambda_{1}$ and
$\Lambda_{2}$ of data functions.
\end{remark}

The main result of this section applies (\ref{eq:L2}) and the corresponding
error bound of \cite{RTPS} for $k=0$ to obtain a $L^{2}$-convergence order for
transfinite surface interpolation with biharmonic Beppo Levi polysplines on
annuli. To state this result, let $W^{2}$ be the Wiener-type algebra of
continuous periodic functions $\mu:\left[  -\pi,\pi\right]  \rightarrow
\mathbb{R}$ with Fourier coefficients $\widehat{\mu}_{k}$, $k\in\mathbb{Z}$,
such that $\sum_{k\in\mathbb{Z}}\left\vert \widehat{\mu}_{k}\right\vert
\left(  1+\left\vert k\right\vert \right)  ^{2}<\infty$. Note that
$W^{2}\subset C^{2}\left[  -\pi,\pi\right]  $ and, as observed in
\cite[Remark~1]{MMA12}, any periodic cubic spline belongs to $W^{2}$.

\begin{theorem}
\label{thm:AO-poly}Given $F\in C^{2}\left(  \mathbb{R}^{2}-\left\{  0\right\}
\right)  \cap C\left(  \mathbb{R}^{2}\right)  $ of polar form $f$ such that
\emph{(\ref{eq:Du-polar})} is finite, assume that $f\left(  r_{j}%
,\cdot\right)  \in W^{2}$ along each domain circle $r=r_{j}$, $j\in\left\{
1,\ldots,n\right\}  $. Let $S$ be either one of the Beppo Levi polysplines
$S^{A}$ or $S^{B}$ determined in \emph{\cite[Theorem~1]{MMA12},} satisfying
the transfinite interpolation conditions \emph{(\ref{eq:transfinite})} for
$\mu_{j}:=f\left(  r_{j},\cdot\right)  $, $j\in\left\{  1,\ldots,n\right\}  $,
where also $S^{A}\left(  0\right)  =F\left(  0\right)  $ and $S^{B}$ is
biharmonic at $0$. Then, for $m\in\left\{  0,1\right\}  $, we have the $L^{2}%
$-error bound%
\begin{equation}
\left(  \int_{r_{1}}^{r_{n}}\int_{-\pi}^{\pi}\left\vert \frac{\partial^{m}%
}{\partial r^{m}}\left(  f-s\right)  \left(  r,\theta\right)  \right\vert
^{2}r\,\mathrm{d}\theta\,\mathrm{d}r\right)  ^{1/2}\leq2^{m-1}\sqrt
{\frac{r_{n}}{r_{1}}}\,h^{2-m}\left\Vert f\right\Vert _{BL}. \label{eq:AO-BL}%
\end{equation}

\end{theorem}

\begin{proof}
For each $r\geq0$, let $\widehat{f}_{k}\left(  r\right)  $, $k\in\mathbb{Z}$,
be the Fourier coefficients of $f\left(  r,\theta\right)  $ with respect to
$\theta$. The smoothness assumptions on $F$ imply that $\widehat{f}_{k}\in
C^{2}\left(  0,\infty\right)  $, $\widehat{f}_{k}$ is continuous at $r=0$,
$\forall k\in\mathbb{Z}$, and identity (\ref{eq:Plancherel}) holds. Since
$\frac{\partial^{m}}{\partial r^{m}}\left(  f-s\right)  \left(  r,\cdot
\right)  \in C\left[  -\pi,\pi\right]  \subset L^{2}\left[  -\pi,\pi\right]
$, the following Plancherel formula is also valid for $m\in\left\{
0,1\right\}  $ and $r\in\left[  r_{1},r_{n}\right]  $:%
\[
\frac{1}{2\pi}\int_{-\pi}^{\pi}\left\vert \frac{\partial^{m}}{\partial r^{m}%
}\left(  f-s\right)  \left(  r,\theta\right)  \right\vert ^{2}\mathrm{d}%
\theta=\sum_{k\in\mathbb{Z}}\left\vert \frac{d^{m}}{dr^{m}}\left(  \widehat
{f}_{k}-\widehat{s}_{k}\right)  \left(  r\right)  \right\vert ^{2}.
\]
Moreover, since $\sqrt{r}\,\frac{\partial^{m}}{\partial r^{m}}\left(
f-s\right)  \in C\left(  \left[  r_{1},r_{n}\right]  \times\left[  -\pi
,\pi\right]  \right)  \subset L^{2}\left(  \left[  r_{1},r_{n}\right]
\times\left[  -\pi,\pi\right]  \right)  $, we may multiply the above relation
by $r$ and integrate both sides to obtain, via Fubini's theorem,%
\begin{equation}
\frac{1}{2\pi}\int_{r_{1}}^{r_{n}}\int_{-\pi}^{\pi}\left\vert \frac
{\partial^{m}}{\partial r^{m}}\left(  f-s\right)  \left(  r,\theta\right)
\right\vert ^{2}r\,\mathrm{d}\theta\,\mathrm{d}r=\sum_{k\in\mathbb{Z}}%
\int_{r_{1}}^{r_{n}}\left\vert \frac{d^{m}}{dr^{m}}\left(  \widehat{f}%
_{k}-\widehat{s}_{k}\right)  \left(  r\right)  \right\vert ^{2}r\,\mathrm{d}r.
\label{eq:Fubini}%
\end{equation}

Note that, for each $j\in\left\{  1,\ldots,n\right\}  $, the transfinite
interpolation condition $s\left(  r_{j},\theta\right)  =f\left(  r_{j}%
,\theta\right)  $, $\forall\theta\in\left[  -\pi,\pi\right]  $, is equivalent
to $\widehat{s}_{k}\left(  r_{j}\right)  =\widehat{f}_{k}\left(  r_{j}\right)
$, $\forall k\in\mathbb{Z}$. Hence, for $k\not =0$, the error bound
(\ref{eq:L2}) implies, for $m\in\left\{  0,1\right\}  $,%
\begin{equation}
\left\Vert \frac{d^{m}}{dr^{m}}\left(  \widehat{f}_{k}-\widehat{s}_{k}\right)
\right\Vert _{L^{2}\left(  \left[  r_{1},r_{n}\right]  \right)  }\leq\frac
{1}{2^{1-m}\sqrt{r_{1}}}h^{2-m}\left\Vert \widehat{f}_{k}\right\Vert _{k}.
\label{eq:L2-k}%
\end{equation}
In addition, it follows from \cite[Theorem~6]{RTPS} that this error bound also
holds for $k=0$, since $\widehat{s}_{0}^{A}\left(  0\right)  =\widehat{f}%
_{0}\left(  0\right)  =F\left(  0\right)  $ and $\widehat{s}_{0}^{B}%
\in\mathrm{span}\left\{  r^{2},1\right\}  $ for $r\in\left(  0,r_{1}\right)  $.

Therefore (\ref{eq:Fubini}), (\ref{eq:L2-k}), and (\ref{eq:Plancherel}) imply%
\begin{eqnarray*}
\lefteqn{  \frac{1}{2\pi}\int_{r_{1}}^{r_{n}}\int_{-\pi}^{\pi}\left\vert
\frac{\partial^{m}}{\partial r^{m}}\left(  f-s^{A,B}\right)  \left(
r,\theta\right)  \right\vert ^{2}r\,\mathrm{d}\theta\,\mathrm{d}r } \\
&  \leq & r_{n}\sum_{k\in\mathbb{Z}}\int_{r_{1}}^{r_{n}}\left\vert \frac{d^{m}%
}{dr^{m}}\left(  \widehat{f}_{k}-\widehat{s}_{k}\right)  \left(  r\right)
\right\vert ^{2}\mathrm{d}r\\
&  \leq & Ch^{2\left(  2-m\right)  }\sum_{k\in\mathbb{Z}}\left\Vert \widehat
{f}_{k}\right\Vert _{k}^{2}=\frac{C}{2\pi}h^{2\left(  2-m\right)  }\left\Vert
f\right\Vert _{BL}^{2},
\end{eqnarray*}
where $C=2^{2\left(  m-1\right)  }r_{n}/r_{1}$, which establishes
(\ref{eq:AO-BL}).\medskip
\qed
\end{proof}

A similar approximation order for transfinite interpolation via biharmonic
Beppo Levi polysplines on parallel strips has recently been proved in
\cite{IMA13}. Related Plancherel representations of the error have been
employed before by Kounchev and Render \cite{KR03} for cardinal polysplines on
annuli and by Sharon and Dyn \cite{SD12} for interpolatory subdivision schemes.


\begin{thebibliography}{99}

\bibitem {ANW67}J.H. Ahlberg, E.N. Nilson, J.L. Walsh, The Theory of Splines
and Their Applications, Academic Press, New York, 1967.

\bibitem {MSc}R.S. Al-Sahli, $L$-spline Interpolation and Biharmonic
Polysplines on Annuli, MSc Thesis, Kuwait University (2012).

\bibitem {JAT08}A. Bejancu, Semi-cardinal polyspline interpolation with Beppo
Levi boundary conditions, J. Approx.\ Theory 155 (2008), 52--73.

\bibitem {CA11}A. Bejancu, Transfinite thin plate spline interpolation,
Constr.\ Approx.\ 34 (2011), 237--256.

\bibitem {MMA12}A. Bejancu, Thin plate splines for transfinite interpolation
at concentric circles, Math.\ Model.\ Anal.\ 18 (2013), 446--460.

\bibitem {IMA13}A. Bejancu, Beppo Levi polyspline surfaces, in: The
Mathematics of Surfaces XIV, R.J. Cripps, G. Mullineux, M.A. Sabin (Eds.), The
Institute of Mathematics and its Applications, UK, 2013; ISBN 978-0-905091-30-3.

\bibitem {RTPS}A. Bejancu, Radially symmetric thin plate splines interpolating
a circular contour map, J. Comput.\ Appl.\ Math.\ 292 (2016), 7--22.

\bibitem {BS88}C. Bennett, R. Sharpley, Interpolation of Operators, Academic
Press, Boston, 1988.

\bibitem {BRS13}M. Bozzini, M. Rossini, R. Schaback, Generalized Whittle-Mat\'{e}rn
and polyharmonic kernels, Adv.\ Comput.\ Math.\ 39 (2013), 129--141.

\bibitem {CFM14}R. Cavoretto, G.E. Fasshauer, M.J. McCourt, An introduction to
the Hilbert-Schmidt SVD using iterated Brownian bridge kernels,
Numer.\ Alg.\ 68 (2015), 393--422.

\bibitem {Du76}J. Duchon, Interpolation des fonctions de deux variables
suivant le principe de la flexion des plaques minces, RAIRO Anal.\ Numer.\ 10
(1976), 5--12.

\bibitem {JP72}J. Jerome, J. Pierce, On spline functions determined by
singular self-adjoint differential operators, J. Approx.\ Theory 5 (1972), 15--40.

\bibitem {KZ66}S. Karlin, Z. Ziegler, Tchebycheffian spline functions, SIAM J.
Numer.\ Anal.\ 3 (1966), 514--543.

\bibitem {Ku01}O.I. Kounchev, Multivariate Polysplines. Applications to
Numerical and Wavelet Analysis, Academic Press, London, 2001.

\bibitem {KR03}O.I. Kounchev, H. Render, The approximation order of
polysplines, Proc.\ Am.\ Math.\ Soc.\ 132 (2003), 455--461.

\bibitem {KR05}O.I. Kounchev, H. Render, Cardinal interpolation with
polysplines on annuli, J. Approx.\ Theory 137 (2005), 89--107.

\bibitem {L70}T.R. Lucas, A generalization of $L$-splines, Numer. Math. 15
(1970), 359--370.

\bibitem {Rab96}C. Rabut, Interpolation with radially symmetric thin plate
splines, J. Comput.\ Appl.\ Math.\ 73 (1996), 241--256.

\bibitem {SV67}M.H. Schultz \& R.S. Varga, $L$-splines, Numer. Math. 10
(1967), 345--369.

\bibitem {LLS07}L.L. Schumaker, Spline Functions: Basic Theory, third ed.,
Cambridge University Press, Cambridge, 2007.

\bibitem {SD12}N. Sharon, N. Dyn, Bivariate interpolation based on univariate
subdivision schemes, J. Approx.\ Theory 164 (2012), 709--730.

\bibitem {W05}H. Wendland, Scattered Data Approximation. Cambridge University
Press, Cambridge, 2005.

\end{thebibliography}
\end{document}